\documentclass[12pt,a4paper]{article}

\usepackage[british]{babel}

\usepackage[a4paper,top=2cm,bottom=2cm,left=2.5cm,right=2.5cm,marginparwidth=1.75cm]{geometry}
\usepackage{graphicx} 
\usepackage{amsrefs}
\usepackage{hyperref}
\usepackage{enumitem}
\usepackage{amsmath, amssymb, amsfonts, amsthm, mathrsfs, enumitem, mathrsfs}
\usepackage{graphicx}
\usepackage[title]{appendix}
\usepackage{mathrsfs}
\usepackage{booktabs}
\usepackage{caption, subcaption}
\usepackage{threeparttable}
\usepackage{algorithm, algorithmicx, algpseudocode}
\usepackage{listings}
\usepackage{chngcntr}
\usepackage{lipsum}
\usepackage{authblk}
\usepackage[utf8]{inputenc}
\usepackage{csquotes}
\usepackage{diagbox}
\usepackage{multirow}
\usepackage{colortbl, xcolor}
\usepackage{epigraph}
\usepackage{tikz}
\usepackage{float}
\usepackage{longtable}
\usepackage{pdfpages, pdflscape}
\usepackage{comment}
\usepackage{CJKutf8}
\usepackage{setspace}
\usepackage{titlesec}
\titleformat{\section}{\normalfont\Large\bfseries}{\thesection.}{1em}{}

\usepackage{lineno}

\rightlinenumbers

\theoremstyle{plain}
\newtheorem{thm}{Theorem}[section]
\newtheorem{prop}[thm]{Proposition}

\theoremstyle{definition}
\newtheorem{defn}{Definition}[section]



\title{Weight Enumerators of codes over $\mathbb{F}_2$ and over  $\mathbb{Z}_4$}

\author[1,2*]{A.K.M. Selim Reza}
\author[2]{Manabu Oura}
\author[3]{Nur Hamid}

\affil[1]{\small Department of Mathematics, Khulna University of Engineering \& Technology, Bangladesh}
\affil[2]{Graduate School of Natural Science and Technology, Kanazawa University, Japan}
\affil[3]{Mathematics Education, Nurul Jadid University, Paiton, Probolinggo, Indonesia}
\affil[*]{Corresponding author: \texttt{selim\_1992@math.kuet.ac.bd}}

\date{}  

\begin{document}

\maketitle

\begin{abstract}
Weight enumerators are important tools for deciphering the algebraic structure of the related code spaces and for understanding group actions on these spaces. Our study focuses on symmetrized weight enumerators of pairs of Type II codes over the finite field $\mathbb{F}_{2}$ and the ring $\mathbb{Z}_{4}$. These pairs have been examined as invariants for a specified group. In particular, we concentrate on the scenarios where the space of the invariant ring is of degree 8 and 16. Our findings show that in certain situations, the ring produced by the symmetrized weight enumerators precisely matches with the invariant ring of the designated group. This coincidence points to a profound relationship between the invariant ring's structure and the algebraic characteristics of the weight enumerators.

\end{abstract}



\section{Introduction}
Following the groundbreaking work of Gleason \cite{gleason1971weight}  the relationships between coding theory, the invariant theory of various finite groups, and modular forms were elucidated by Brou\'{e}-Enguehard \cite{broue1972polynomes}, Duke \cite{duke1993codes}, Runge \cite{runge1996codes}, Bannai et al \cite{bannai1999type}. These investigations demonstrate the relationship between combinatorics and number theory.

The algebraic method of computing rings of Hilbert modular forms for a specific generalized paramodular group was demonstrated by Runge \cite{runge1996symmetric}.

Later, Oura \cite{oura1999codes} constructed a paramodular form using coding theory.
He presented a theoretical argument in his writing.
The assertion was not backed up by any instances. His concepts serve as the basis for our research.
To substantiate his claim, we primarily investigate Type II codes of lengths 8 and 16 over $\mathbb{F}_{2}$ and $\mathbb{Z}_{4}$ and construct weight polynomials of degrees 8 and 16.
It seems that weight enumerators of codes can create the invariant ring of a finite group $H$ mentioned in Section \ref{sec:result}.

Numerous research works have been carried out on Type II codes over $\mathbb{Z}_{2k}$. It was shown in \cite{bannai1999type} that there is a Type II \(\mathbb{Z}_{2k} \) code of length $n$ if and only if $n$ is divisible by 8. 
Klemm \cite{klemm1989selbstduale}, followed by Conway and Sloane \cite{conway1993self}, examined the requirements met by the weight enumerator of self-dual codes defined over the ring of integers modulo four.
In \cite{conway1993self}, Conway and Sloane mentioned four Type II codes, each of type $4^4, 4^32^2, 4^22^3, \text{ and } 4.2^6$ of length 8 including the code $\mathscr{K}_{4m}, m\geq 1$, introduced by Klemm in \cite{klemm1989selbstduale}. 
Gaborit \cite{gaborit1996mass} discovered a mass formula for both the Type II codes, which are only present at lengths divisible by eight.
Later, in \cite{pless1997all}, Pless et al. demonstrated 133 inequivalent Type II $\mathbb{Z}_{4}$ codes of length 16 using a new computer algorithm.
Ten of these codes are direct sums of the four Type II codes of \cite{conway1993self} and the remaining codes are discovered to be indecomposable. The generator matrices of the codes that we employ in \ref{sec:result} are available in \cite{conway1993self}. We refer to \cite{nebe2006self} for the definition used in the section \ref{sec:preliminaries}. 

Throughout the paper, we use $\mathbb{Z}, \mathbb{Q}, \mathbb{C}$ as the sets of integers, rationals, and complex numbers respectively. We put \(\mathbb{Z}_{l} = \mathbb{Z}/{l\mathbb{Z}}\). Though, $\mathbb{Z}_{2} \text{ and } \mathbb{F}_{2}$ represent the finite field with elements ${0,1}$ and the ring of integer modulo $2$ respectively, they are same in terms of structure and in coding theory they are interchangeable , hence we consider \( \mathbb{Z}_{2} = \mathbb{F}_{2}\). We sometimes regard \( \mathbb{Z}_{l} = \{0,1,\ldots,l-1\}\) as integers.
We denote \( exp({2\pi \sqrt{-1}}\; \cdot) \) by \( \mathbf{e}(\cdot) \), \( A[B] = B^T A B \) where \( B^T \) denotes the transpose of \( B \), and $\delta_{ij}$ is Kronecker delta which has its usual meaning. 

We provide definitions for all terminology in Section \ref{sec:preliminaries} in order to prevent ambiguities and enhance understanding of the subjects covered in later parts. We present our primary findings in Section \ref{sec:result}, together with a detailed analysis and a focus on the important discoveries that support the goals of our investigation as a whole.

\section{Preliminaries} \label{sec:preliminaries}
The direct product of \(n\) copies of the cyclic group \(\mathbb{Z}_{2k}\) is denoted by \(\mathbb{Z}_{2k}^n\).
A subset of \( \mathbb{Z}^n_{2k} \) is called a linear code \( \mathscr{C} \) of length \( n \)  if it forms an additive subgroup of \( \mathbb{Z}^n_{2k} \). An element of $\mathscr{C}$ is called a codeword of $\mathscr{C}$. A generator matrix of $\mathscr{C}$ is a matrix whose rows generate $\mathscr{C}$.   If $ x = (x_1, x_2, \ldots, x_n) \text{ and } y = (y_1, y_2, \ldots, y_n) \in \mathbb{Z}_{2k}^n $ then the inner product of $x$ and $y$ is defined as follows:
\[
\langle x, y \rangle = x_1 y_1 + x_2 y_2 + \cdots + x_n y_n \pmod{2k}.
\]

The dual code of \( \mathscr{C} \), denoted by \( \mathscr{C}^\perp \), is defined as \( \mathscr{C}^\perp = \{ x \in \mathbb{Z}^n_{2k} \mid \langle x, y \rangle = 0 \pmod{2k} \quad  \forall  y \in \mathscr{C}\} \).
If  \( \mathscr{C} = \mathscr{C}^\perp \) then \( \mathscr{C} \) is said to be self dual. 
We say \( \mathscr{C}\) is Type II if \( \mathscr{C}\) is self dual and  \( \langle x, x \rangle = 0 \pmod{4k} \) for any element \( x \in \mathscr{C} \).

\begin{defn}
    An invariant ring of a finite subgroup \( G \) of $GL(n,\mathbb{C})$ acting on a polynomial ring \( \mathbb{C}[x_1, x_2, \ldots, x_n] \) over $\mathbb{C}$ consists of all polynomials that remain unchanged under the action of every element of \( G \).
    It is denoted by \( \mathbb{C}[x_1, x_2, \ldots, x_n]^G \).
    The action is defined as  
$$
A\cdot f(x_1, x_2, \ldots, x_n) = f\left(\sum_{1 \leq j \leq n} A_{1j} x_j, \ldots, \sum_{1 \leq j \leq n} A_{nj} x_j \right)
$$
where  \( f  \in \mathbb{C}[x_1, x_2, \ldots, x_n] \) and $A= (A_{ij})_{1\leq i,j \leq n}$.
\end{defn}

There exists a homogeneous system of parameter ${\theta_1,\; \theta_2, \ldots, \theta_n}$ such that the invariant ring \( \mathbb{C}[x_1, x_2, \ldots, x_n]^G \) is finitely generated free $\mathbb{C}[\theta_1,\; \theta_1, \ldots, \theta_n]$- module.
It has the Hironaka decomposition, \( \displaystyle \mathbb{C}[x_1, x_2, \ldots, x_n]^G =  \bigoplus_{1 \leq m \leq s}g_m\mathbb{C}[\theta_1,\; \theta_2, \ldots, \theta_n]\)  with $g_1 = 1$.
This invariant ring is a graded ring and its dimension formula is defined by $\phi_G(t)= \sum_{d\geq 1}\dim\mathbb{C}[x_1, x_2, \ldots, x_n]_d^G\;t^d $ where  $\mathbb{C}[x_1, x_2, \ldots, x_n]_d^G$ represents the number of linearly independent invariant polynomials of degree \( d \).
The dimension formula for the Hironaka decomposition given in the above form is 
$$
\phi_G(t)=\sum_{A\in G}\frac{1+t^{\text{deg}{(g_2)}}+\ldots +t^{\text{deg}{(g_s)}}} {(1-t^{\text{deg}{(\theta_1)}})\ldots (1-t^{\text{deg}{(\theta_n)}})}.
$$
\noindent By Molien it is given by
\[
\phi_G(t) = \frac{1}{|G|} \sum_{\lambda \in G} \frac{1}{\det(I - t \cdot \lambda)}
\] 
where \( I \) is the identity matrix and \( t \) is a formal variable.

Let \( k_1, k_2, \ldots, k_g \) be an ordered sequence of positive integers such that \( k_i \mid k_{i+1} \) for \( i = 1, 2, \ldots, g-1 \) and \( R = \mathbb{Z}_{2k_1} \times \mathbb{Z}_{2k_2} \times \cdots \times \mathbb{Z}_{2k_g} \) and let \(\alpha\) be a positive integer such that \(\alpha = k_1\) and \(\alpha_i = \frac{k_i}{\alpha} \) for \( i = 1, 2, \ldots, g \). We set \( D = \text{diag} (\alpha_1, \alpha_2, \ldots, \alpha_g) \). Then we can form a subgroup  $G$ of \( GL\left(2^{g}(\alpha^g \alpha_1 \alpha_2 \ldots \alpha_{g}), \mathbb{C}\right) \) whose generators are \( \chi_\alpha,\; \xi_U, \;\eta_S, \; \; \zeta_n \) and it operates on $\mathbb{C}[x_a \mid a \in R ]$. We consider the generator \( \zeta_n = \mathbf{e}(1/8)*I_n \), where $I_n$ represents identity matrix whose order is as equal as others generators of $G$. We define the other generators of $ G $ using the following conditions:  
\begin{equation}
\label{eq:T}
\chi_\alpha = \mathbf{e}\left(1/8\right)^g \left(2^g \alpha^g \alpha_1 \alpha_2 \dots \alpha_g\right)^{-1/2} \left(\mathbf{e}\left( \langle D^{-1}a,b \rangle /2\alpha\right)\right)_{a , b \in R }
\end{equation}

Let \(\Omega(D)\) be a subgroup of \(\mathrm{GL}(g, \mathbb{Z})\) consisting with the set of elements \(U \in \mathrm{GL}(g, \mathbb{Z})\) such that \(D^{-1}AD\) is an integer matrix. We set $\xi_U$ as follows:
\begin{align}
&\Omega(D) = \left\{ (a_{ij}) \in \mathrm{GL}(g, \mathbb{Z}) \mid a_{ji} \equiv 0 \pmod{( \alpha_i/\alpha_j )} \text{ for } 1 \leq j < i \leq g \right\} \nonumber \\
& \xi_{U} = \left(\sqrt{\det(U)} \, \delta_{U_{a,b}}\right)_{a,b \in R} \quad \text{for } U \in \Omega(D) \label{eq:Omega}
\end{align}

Let \(\Lambda(D)\) be the set of symmetric elements \( S \in \mathrm{Mat}(g, \mathbb{Q}) \) such that \( SD \) is an integer matrix and we set \( \eta_S \) as follows: 
\begin{align}
& \Lambda(D) = \left\{ (S_{ij}) \in \mathrm{Mat}(g, \mathbb{Q}) \mid S_{ij} = S_{ji} \in \frac{1}{k_i}\mathbb{Z} \text{ for } 1 \leq i \leq j \leq g \right\} \nonumber \\
& \eta_S = \mathrm{diag} \left( \mathbf{e} ( S[a]/4\alpha) \right) \quad \text{with } a \in R \quad \text{for } S \in \Lambda(D)
\label{eq:Lambda}
\end{align}

\noindent In order to introduce the symmetrized weight enumerators, we use the equivalency relation $a \sim b \Leftrightarrow a = b \text{ or } a = -b $ on \( R \) where \( a, b \in R \). We put \( \bar{R} = R / \sim \). To examine the invariance characteristics of symmetrized weight enumerators we refer to \cite{runge1996symmetric} , \cite{oura1999codes}.
\begin{defn}
    The symmetrized weight enumerators of the codes $ \mathscr{C}_1,\mathscr{C}_2,...,\mathscr{C}_g $ is defined by 
\begin{equation}
\mathscr{W}_{(\mathscr{C}_1, \mathscr{C}_2, \ldots, \mathscr{C}_g)}(x_{\bar{a}} \text{ with } \bar{a} \in \bar{R}) = \sum_{c_1 \in \mathscr{C}_1, \, c_2 \in \mathscr{C}_2, \, \ldots, \, c_g \in \mathscr{C}_g} \prod_{\bar{a} \in \bar{R}} x_{\bar{a}}^{n_{\bar{a}}(c_1, c_2, \ldots, c_g)}  
 \label{eq:SWE}
\end{equation}
where \( n_{\bar{a}}(c_1, c_2, \ldots, c_g) \) designates the number of \( i \) such that \( \bar{a} = \overline{(c_{1i}, c_{2i}, \ldots, c_{gi})} \).
\end{defn}

\noindent Now let \( g = (g_{ab})_{a, b \in R} \) and $\phi$  be a function define as 
\begin{equation}
    \phi(g) = \left( \sum_{d \in R \mid \overline{d} = \overline{b}} g_{ad} \right)_{\overline{a}, \overline{b} \in \overline{R}}
    \label{eq:six}
\end{equation}

Applying the function on equations \eqref{eq:T}, \eqref{eq:Omega}, \eqref{eq:Lambda} we find the symmetrized matrices $\phi(\chi_\alpha),\; \phi(\xi_U),\; \phi(\eta_S) $ respectively. Now we form  $H = \langle \phi(\chi_\alpha),\; \phi(\xi_U),\; \phi(\eta_S),\; \phi(\zeta_n)\rangle $ which is a subgroup of $ GL\left(2^{g-1}(\alpha^g \alpha_1 \alpha_2 \ldots \alpha_{g}+1),\mathbb{C}\right) $ that operates naturally on $\mathbb{C}[x_{\overline{a}}] = \mathbb{C}[x_{\overline{a}} \mid \overline{a} \in \overline{R}]$, where \( U \) and \( S \) travel through \( \Omega(D) \) and \( \Lambda(D) \) respectively.

\section{Results} \label{sec:result}

We consider for \( g = 2 \) and let, \( k_1 = 1 \), \( k_2 = 2 \) to get \( R = \mathbb{F}_2 \times \mathbb{Z}_4 \). We set  \( \alpha = 1 \). Thus we get \( \alpha_1 = 1 \), \( \alpha_2 = 2 \) as diagonal entries of $D$. To obtain the generators of $G$, we consider a few number of matrices  $U\in \Omega(D)$ and $S\in \Lambda (D)$. We use Sage-math \cite{sagemath} to form them. However, for the sake of simplicity, keeping the group order unchanged we simply take the following two matrices from $U$ and $S$. 
$$
u_1 =
\begin{bmatrix}
-1 & -1 \\
0 & -1
\end{bmatrix} 
\text {, }
u_2 =
\begin{bmatrix}
-1 & -1 \\
-2 & -1
\end{bmatrix}
\in \Omega(D) 
\quad
s_1 =
\begin{bmatrix}
1 & 1 \\
1 & 1/2
\end{bmatrix} 
\text { , }
s_2 =
\begin{bmatrix}
0 & 0 \\
0 & 1/2
\end{bmatrix}
\in \Lambda (D)
$$

\noindent With all of the above parameters, we obtain the following $\chi_\alpha$ from \eqref{eq:T}, $\xi{u_1},\; \xi{u_2}$ from \eqref{eq:Omega}, and $\eta_{s_1}, \eta_{s_2}$ from \eqref{eq:Lambda}.
Here $z$ denotes the $8$-th root of unity.\\

\noindent\resizebox{\textwidth}{!}{
\(
\begin{array}{c}
\chi_\alpha=
\begin{bmatrix}
\frac{1}{4}z^3 + \frac{1}{4}z & \frac{1}{4}z^3 + \frac{1}{4}z & \frac{1}{4}z^3 + \frac{1}{4}z & \frac{1}{4}z^3 + \frac{1}{4}z & \frac{1}{4}z^3 + \frac{1}{4}z & \frac{1}{4}z^3 + \frac{1}{4}z & \frac{1}{4}z^3 + \frac{1}{4}z & \frac{1}{4}z^3 + \frac{1}{4}z \\
\frac{1}{4}z^3 + \frac{1}{4}z & \frac{1}{4}z^3 - \frac{1}{4}z & -\frac{1}{4}z^3 - \frac{1}{4}z & -\frac{1}{4}z^3 + \frac{1}{4}z & \frac{1}{4}z^3 + \frac{1}{4}z & \frac{1}{4}z^3 - \frac{1}{4}z & -\frac{1}{4}z^3 - \frac{1}{4}z & -\frac{1}{4}z^3 + \frac{1}{4}z \\
\frac{1}{4}z^3 + \frac{1}{4}z & -\frac{1}{4}z^3 - \frac{1}{4}z & \frac{1}{4}z^3 + \frac{1}{4}z & -\frac{1}{4}z^3 - \frac{1}{4}z & \frac{1}{4}z^3 + \frac{1}{4}z & -\frac{1}{4}z^3 - \frac{1}{4}z & \frac{1}{4}z^3 + \frac{1}{4}z & -\frac{1}{4}z^3 - \frac{1}{4}z \\
\frac{1}{4}z^3 + \frac{1}{4}z & -\frac{1}{4}z^3 + \frac{1}{4}z & -\frac{1}{4}z^3 - \frac{1}{4}z & \frac{1}{4}z^3 - \frac{1}{4}z & \frac{1}{4}z^3 + \frac{1}{4}z & -\frac{1}{4}z^3 + \frac{1}{4}z & -\frac{1}{4}z^3 - \frac{1}{4}z & \frac{1}{4}z^3 - \frac{1}{4}z \\
\frac{1}{4}z^3 + \frac{1}{4}z & \frac{1}{4}z^3 + \frac{1}{4}z & \frac{1}{4}z^3 + \frac{1}{4}z & \frac{1}{4}z^3 + \frac{1}{4}z & -\frac{1}{4}z^3 - \frac{1}{4}z & -\frac{1}{4}z^3 - \frac{1}{4}z & -\frac{1}{4}z^3 - \frac{1}{4}z & -\frac{1}{4}z^3 - \frac{1}{4}z \\
\frac{1}{4}z^3 + \frac{1}{4}z & \frac{1}{4}z^3 - \frac{1}{4}z & -\frac{1}{4}z^3 - \frac{1}{4}z & -\frac{1}{4}z^3 + \frac{1}{4}z & -\frac{1}{4}z^3 - \frac{1}{4}z & -\frac{1}{4}z^3 + \frac{1}{4}z & \frac{1}{4}z^3 + \frac{1}{4}z & \frac{1}{4}z^3 - \frac{1}{4}z \\
\frac{1}{4}z^3 + \frac{1}{4}z & -\frac{1}{4}z^3 - \frac{1}{4}z & \frac{1}{4}z^3 + \frac{1}{4}z & -\frac{1}{4}z^3 - \frac{1}{4}z & -\frac{1}{4}z^3 - \frac{1}{4}z & \frac{1}{4}z^3 + \frac{1}{4}z & -\frac{1}{4}z^3 - \frac{1}{4}z & \frac{1}{4}z^3 + \frac{1}{4}z \\
\frac{1}{4}z^3 + \frac{1}{4}z & -\frac{1}{4}z^3 + \frac{1}{4}z & -\frac{1}{4}z^3 - \frac{1}{4}z & \frac{1}{4}z^3 - \frac{1}{4}z & -\frac{1}{4}z^3 - \frac{1}{4}z & \frac{1}{4}z^3 - \frac{1}{4}z & \frac{1}{4}z^3 + \frac{1}{4}z & -\frac{1}{4}z^3 + \frac{1}{4}z
\end{bmatrix}
\end{array}
\)
}
\vspace{1em}

\noindent\resizebox{\textwidth}{!}{
\(
\begin{array}{cc}
\xi_{u_1}=
\begin{bmatrix}
1 & 0 & 0 & 0 & 0 & 0 & 0 & 0 \\
0 & 0 & 0 & 0 & 0 & 0 & 0 & 1 \\
0 & 0 & 1 & 0 & 0 & 0 & 0 & 0 \\
0 & 0 & 0 & 0 & 0 & 1 & 0 & 0 \\
0 & 0 & 0 & 0 & 1 & 0 & 0 & 0 \\
0 & 0 & 0 & 1 & 0 & 0 & 0 & 0 \\
0 & 0 & 0 & 0 & 0 & 0 & 1 & 0 \\
0 & 1 & 0 & 0 & 0 & 0 & 0 & 0
\end{bmatrix}
&
\quad 
\xi_{u_2}=
\begin{bmatrix}
z^2 & 0 & 0 & 0 & 0 & 0 & 0 & 0 \\
0 & 0 & 0 & 0 & 0 & 0 & 0 & z^2 \\
0 & 0 & z^2 & 0 & 0 & 0 & 0 & 0 \\
0 & 0 & 0 & 0 & 0 & z^2 & 0 & 0 \\
0 & 0 & 0 & 0 & 0 & 0 & z^2 & 0 \\
0 & z^2 & 0 & 0 & 0 & 0 & 0 & 0 \\
0 & 0 & 0 & 0 & z^2 & 0 & 0 & 0 \\
0 & 0 & 0 & z^2 & 0 & 0 & 0 & 0 \\
\end{bmatrix}
\end{array}
\)
}
\vspace{1em}

\noindent\resizebox{\textwidth}{!}{
\(
\begin{array}{cc}
\eta_{s_1}=
\begin{bmatrix}
1 & 0 & 0 & 0 & 0 & 0 & 0 & 0 \\
0 & z & 0 & 0 & 0 & 0 & 0 & 0 \\
0 & 0 & -1 & 0 & 0 & 0 & 0 & 0 \\
0 & 0 & 0 & z & 0 & 0 & 0 & 0 \\
0 & 0 & 0 & 0 & z^2 & 0 & 0 & 0 \\
0 & 0 & 0 & 0 & 0 & -z^3 & 0 & 0 \\
0 & 0 & 0 & 0 & 0 & 0 & -z^2 & 0 \\
0 & 0 & 0 & 0 & 0 & 0 & 0 & -z^3 \\
\end{bmatrix}
&
\quad 
\eta_{s_2}=
\begin{bmatrix}
1 & 0 & 0 & 0 & 0 & 0 & 0 & 0 \\
0 & z & 0 & 0 & 0 & 0 & 0 & 0 \\
0 & 0 & -1 & 0 & 0 & 0 & 0 & 0 \\
0 & 0 & 0 & z & 0 & 0 & 0 & 0 \\
0 & 0 & 0 & 0 & 1 & 0 & 0 & 0 \\
0 & 0 & 0 & 0 & 0 & z & 0 & 0 \\
0 & 0 & 0 & 0 & 0 & 0 & -1 & 0 \\
0 & 0 & 0 & 0 & 0 & 0 & 0 & z \\
\end{bmatrix}
\end{array}
\)
}
\vspace{1 em}

Now, we form the subgroup $G$ as $ G = \langle \chi_\alpha,\; \xi_{u_1},\; \xi_{u_2},\; \eta_{s_1},\; \eta_{s_2},\; \zeta_8\rangle$  and \( \left| G \right | = 589824\). Applying the function defined in \eqref{eq:six} on \( R \) we calculate the generators of H whose are given below and form $ H = \langle \phi(\chi_\alpha),\; \phi(\xi_{u_1}),\; \phi(\xi_{u_2}),\; \phi(\eta_{s_1}),\; \phi(\eta_{s_2}),\; \phi(\zeta_8) \rangle $, and find \( \left| H \right | = 294912\). \\

\noindent\resizebox{\textwidth}{!}{
\(
\begin{array}{c}
\phi(\chi_\alpha ) =
\begin{bmatrix}
\frac{1}{4}z^3 + \frac{1}{4}z & \frac{1}{2}z^3 + \frac{1}{2}z & \frac{1}{4}z^3 + \frac{1}{4}z & \frac{1}{4}z^3 + \frac{1}{4}z & \frac{1}{2}z^3 + \frac{1}{2}z & \frac{1}{4}z^3 + \frac{1}{4}z \\
\frac{1}{4}z^3 + \frac{1}{4}z & 0 & -\frac{1}{4}z^3 - \frac{1}{4}z & \frac{1}{4}z^3 + \frac{1}{4}z & 0 & -\frac{1}{4}z^3 - \frac{1}{4}z \\
\frac{1}{4}z^3 + \frac{1}{4}z & -\frac{1}{2}z^3 - \frac{1}{2}z & \frac{1}{4}z^3 + \frac{1}{4}z & \frac{1}{4}z^3 + \frac{1}{4}z & -\frac{1}{2}z^3 - \frac{1}{2}z & \frac{1}{4}z^3 + \frac{1}{4}z \\
\frac{1}{4}z^3 + \frac{1}{4}z & \frac{1}{2}z^3 + \frac{1}{2}z & \frac{1}{4}z^3 + \frac{1}{4}z & -\frac{1}{4}z^3 - \frac{1}{4}z & -\frac{1}{2}z^3 - \frac{1}{2}z & -\frac{1}{4}z^3 - \frac{1}{4}z \\
\frac{1}{4}z^3 + \frac{1}{4}z & 0 & -\frac{1}{4}z^3 - \frac{1}{4}z & -\frac{1}{4}z^3 - \frac{1}{4}z & 0 & \frac{1}{4}z^3 + \frac{1}{4}z \\
\frac{1}{4}z^3 + \frac{1}{4}z & -\frac{1}{2}z^3 - \frac{1}{2}z & \frac{1}{4}z^3 + \frac{1}{4}z & -\frac{1}{4}z^3 - \frac{1}{4}z & \frac{1}{2}z^3 + \frac{1}{2}z & -\frac{1}{4}z^3 - \frac{1}{4}z
\end{bmatrix}
\end{array}
\)
}
\vspace{1 em}

\noindent\resizebox{\textwidth}{!}{
\(
\begin{array}{cc}
\phi(\xi_{u_1}) =
\begin{bmatrix}
1 & 0 & 0 & 0 & 0 & 0 \\
0 & 0 & 0 & 0 & 1 & 0 \\
0 & 0 & 1 & 0 & 0 & 0 \\
0 & 0 & 0 & 1 & 0 & 0 \\
0 & 1 & 0 & 0 & 0 & 0 \\
0 & 0 & 0 & 0 & 0 & 1
\end{bmatrix}
&
\quad 
\phi(\xi_{u_2}) =
\begin{bmatrix}
z^2 & 0 & 0 & 0 & 0 & 0 \\
0 & 0 & 0 & 0 & z^2 & 0 \\
0 & 0 & z^2 & 0 & 0 & 0 \\
0 & 0 & 0 & 0 & 0 & z^2 \\
0 & z^2 & 0 & 0 & 0 & 0 \\
0 & 0 & 0 & z^2 & 0 & 0
\end{bmatrix}
\end{array}
\)
}
\vspace{1 em}

\noindent\resizebox{\textwidth}{!}{
\(
\begin{array}{cc}
\phi(\eta_{s_1}) =
\begin{bmatrix}
1 & 0 & 0 & 0 & 0 & 0 \\
0 & z & 0 & 0 & 0 & 0 \\
0 & 0 & -1 & 0 & 0 & 0 \\
0 & 0 & 0 & z^2 & 0 & 0 \\
0 & 0 & 0 & 0 & -z^3 & 0 \\
0 & 0 & 0 & 0 & 0 & -z^2
\end{bmatrix}
&
\quad 
\phi(\eta_{s_2}) =
\begin{bmatrix}
1 & 0 & 0 & 0 & 0 & 0 \\
0 & z & 0 & 0 & 0 & 0 \\
0 & 0 & -1 & 0 & 0 & 0 \\
0 & 0 & 0 & 1 & 0 & 0\\
0 & 0 & 0 & 0 & z & 0 \\
0 & 0 & 0 & 0 & 0 & -1
\end{bmatrix}
\end{array}
\)
}

\begin{prop}
    The dimension formula for the group $H$ is as follows 
    \begin{align}
    \Phi_{H}(t) &= \frac{1 - t^8 + 3t^{16} + 4t^{24} + 5t^{32} + 3t^{40} + 7t^{48} + 2t^{56}}{(1-t^8)^3(1-t^{24})^3}  \nonumber \\
    &= 1 + 2t^{8} + 6t^{16} + 20t^{24} + 46t^{32} + 96t^{40} + 195t^{48} + \cdots.
    \label{eq:phiG6}  
    \end{align}
\end{prop}
Let \( \mathfrak{R} = \mathbb{C}[a,b,c,d,e,f]^H \) represents the invariant ring of \( H \). From the dimension formula of \( H \), we have the following theorems.

\begin{thm}
    The symmetrized weight enumerators $\mathscr{W}_{(\mathscr{E}_8,\mathscr{Q}_8)}$ and $\mathscr{W}_{(\mathscr{E}_8,\mathscr{K}_8)}$ can be considered as the basis of the subspace of degree 8 of $\mathfrak{R}$.
\end{thm}

\begin{proof}

According to the dimension formula, two polynomials of degree 8 exist in $\mathfrak{R}$. To find them from codes, we use $\mathscr{E}_8$ as the code over $\mathbb{F}_2$ and $\mathscr{Q}_8$, $\mathscr{K}_8$ as codes over $\mathbb{Z}_4$. These codes have the following generator matrices. \\
\resizebox{\textwidth}{!}{
\(
\begin{array}{ccc}
\mathscr{E}_8 : 
\begin{bmatrix}
1 & 1 & 1 & 1 & 0 & 0 & 0 & 0 \\
0 & 0 & 1 & 1 & 1 & 1 & 0 & 0 \\
0 & 0 & 0 & 0 & 1 & 1 & 1 & 1 \\
1 & 0 & 1 & 0 & 1 & 0 & 1 & 0
\end{bmatrix}
&
\quad 
\mathscr{Q}_8 :  
\begin{bmatrix}
0 & 0 & 1 & 1 & 0 & 2 & 1 & 3 \\
0 & 0 & 0 & 2 & 1 & 3 & 1 & 1 \\
1 & 1 & 0 & 2 & 0 & 0 & 1 & 3 \\
0 & 2 & 0 & 2 & 0 & 2 & 0 & 2 \\
0 & 0 & 0 & 0 & 0 & 0 & 2 & 2
\end{bmatrix}
&
\quad 
\mathscr{K}_8 : 
\begin{bmatrix}
1 & 1 & 1 & 1 & 1 & 1 & 1 & 1 \\
0 & 2 & 0 & 0 & 0 & 0 & 0 & 2 \\
0 & 0 & 2 & 0 & 0 & 0 & 0 & 2 \\
0 & 0 & 0 & 2 & 0 & 0 & 0 & 2 \\
0 & 0 & 0 & 0 & 2 & 0 & 0 & 2 \\
0 & 0 & 0 & 0 & 0 & 2 & 0 & 2 \\
0 & 0 & 0 & 0 & 0 & 0 & 2 & 2
\end{bmatrix}
\end{array}
\)
}

We use the variables $ a,b,c,d,e,f $ instead of \(x_{\underset{0}{0}}, x_{\underset{1}{0}}, x_{\underset{2}{0}}, x_{\underset{0}{1}},x_{\underset{1}{1}}, x_{\underset{2}{1}}\) in \eqref{eq:SWE} and calculate the symmetrized weight enumerators $\mathscr{W}_{(\mathscr{E}_8,\mathscr{Q}_8)} \text{ and } \mathscr{W}_{(\mathscr{E}_8,\mathscr{K}_8)}$ as follows. 
\vspace{0.5cm}

\(\mathscr{W}_{(\mathscr{E}_8,\mathscr{Q}_8)} = a^{8} + 32 b^{8} + 96 a^{3} b^{4} c + 4 a^{6} c^{2} + 96 a b^{4} c^{3} + 22 a^{4} c^{4} + 4 a^{2} c^{6} + c^{8} + 14 a^{4} d^{4} + 12 a^{2} c^{2} d^{4} + 14 c^{4} d^{4} + d^{8} + 576 a b^{2} c d^{2} e^{2} + 448 b^{4} e^{4} + 96 a^{3} c e^{4} + 96 a c^{3} e^{4} + 32 e^{8} + 96 b^{4} d^{3} f + 32 a^{3} c d^{3} f + 32 a c^{3} d^{3} f + 576 a^{2} b^{2} d e^{2} f + 576 b^{2} c^{2} d e^{2} f + 96 d^{3} e^{4} f + 12 a^{4} d^{2} f^{2} + 216 a^{2} c^{2} d^{2} f^{2} + 12 c^{4} d^{2} f^{2} + 4 d^{6} f^{2} + 576 a b^{2} c e^{2} f^{2} + 96 b^{4} d f^{3} + 32 a^{3} c d f^{3} + 32 a c^{3} d f^{3} + 96 d e^{4} f^{3} + 14 a^{4} f^{4} + 12 a^{2} c^{2} f^{4} + 14 c^{4} f^{4} + 22 d^{4} f^{4} + 4 d^{2} f^{6} + f^{8}\);\\

\(\mathscr{W}_{(\mathscr{E}_8,\mathscr{K}_8)}= a^{8} + 128 b^{8} + 28 a^{6} c^{2} + 70 a^{4} c^{4} + 28 a^{2} c^{6} + c^{8} + 14 a^{4} d^{4} + 84 a^{2} c^{2} d^{4} + 14 c^{4} d^{4} + d^{8} + 1792 b^{4} e^{4} + 128 e^{8} + 224 a^{3} c d^{3} f + 224 a c^{3} d^{3} f + 84 a^{4} d^{2} f^{2} + 504 a^{2} c^{2} d^{2} f^{2} + 84 c^{4} d^{2} f^{2} + 28 d^{6} f^{2} + 224 a^{3} c d f^{3} + 224 a c^{3} d f^{3} + 14 a^{4} f^{4} + 84 a^{2} c^{2} f^{4} + 14 c^{4} f^{4} + 70 d^{4} f^{4} + 28 d^{2} f^{6} + f^{8}\);\\

They are invariants under the action of $H$ by magma \cite{magma}. To ascertain whether $\mathscr{W}_{(\mathscr{E}_8,\mathscr{Q}_8)} \text{ and } \mathscr{W}_{(\mathscr{E}_8,\mathscr{K}_8)}$ are linearly independent, we consider the coefficients of the monomials $ a^8,\; b^8 $ and find a matrix of nonzero determinant. \\
Furthermore, we can directly compute the two bases of degree 8 polynomials using magma \cite{magma}, and express them in terms of $\mathscr{W}_{(\mathscr{E}_8,\mathscr{Q}_8)} \text{ and } \mathscr{W}_{(\mathscr{E}_8,\mathscr{K}_8)}$. 

\end{proof}

\begin{thm}
       The symmetrized weight enumerators $\mathscr{W}^2_{(\mathscr{E}_8,\mathscr{Q}_8)}$, $\mathscr{W}^2_{(\mathscr{E}_8,\mathscr{K}_8)}$, 
        $\mathscr{W}_{(\mathscr{E}_8\oplus\mathscr{E}_8,\mathscr{Q}_8 \oplus \mathscr{K}_8)}$, $\mathscr{W}_{(\mathscr{E}_8\oplus\mathscr{E}_8, \mathscr{K}_{16})}$, 
        $\mathscr{W}_{(\mathscr{D}_{16},\mathscr{Q}_8 \oplus \mathscr{Q}_8)}$, and $\mathscr{W}_{(\mathscr{D}_{16},\mathscr{Q}_8 \oplus \mathscr{K}_8)}$ 
        can be treated as the basis of the subspace of degree 16 of $\mathfrak{R}$. 
\end{thm}
 
\begin{proof}

To find the six invariant polynomials of degree 16, we use direct sum of the length 8 codes used in previous theorem along with $\mathscr{D}_{16}$ as the code over $\mathbb{F}_2$ \text{ and } $\mathscr{K}_{16}$ as the code over $\mathbb{Z}_4$. We give the generator matrices of \(\mathscr{D}_{16}\text{ and } \mathscr{K}_{16}\) in appendix. We calculate the symmetrized weight enumerators $\mathscr{W}^2_{(\mathscr{E}_8,\mathscr{Q}_8)}$, $\mathscr{W}^2_{(\mathscr{E}_8,\mathscr{K}_8)}$, $\mathscr{W}_{(\mathscr{E}_8\oplus\mathscr{E}_8,\mathscr{Q}_8 \oplus \mathscr{K}_8)}$, $\mathscr{W}_{(\mathscr{E}_8\oplus\mathscr{E}_8, \mathscr{K}_{16})}$, $\mathscr{W}_{(\mathscr{D}_{16},\mathscr{Q}_8 \oplus \mathscr{Q}_8)}$, and $\mathscr{W}_{(\mathscr{D}_{16},\mathscr{Q}_8 \oplus \mathscr{K}_8)}$. We do not provide those weight enumerators here because they are too big; instead, the reader can find them in \cite{selim}. By Magma all of them are invariants under the action of $H$. To verify their linear independence, we pick up the following monomials \[ac^3d^3f^9, a^2c^2d^4f^8, c^4d^4f^8, d^8f^8, ab^2cd^2e^2f^8, b^4e^4f^8 \] and using the coefficient of them we obtain a non-singular matrix which proves that all six weight enumerators are linearly independent. Furthermore, six bases of degree 16 of \(\mathfrak{R}\) can be identified directly using magma, and those can be expressed in terms of these six weight enumerators.

\end{proof}

\section*{Acknowledgments}
The first named author of this work was supported in part by funds from  Ministry of Education, Culture, Sports, Science and Technology, Japan and the second named author was supported by JSPS KAKENHI Grant number 24K06827.


\section*{Appendix}

The generating matrices of \(\mathscr{D}_{16}\text{ and } \mathscr{K}_{16}\) are as follows. 
\[
\begingroup 
\mathscr{D}_{16} = 
\setlength\arraycolsep{1pt}
\left(
\begin{array}{*{16}{c}}
1&1&1&1&0&0&0&0&0&0&0&0&0&0&0&0 \\
0&0&1&1&1&1&0&0&0&0&0&0&0&0&0&0 \\
0&0&0&0&1&1&1&1&0&0&0&0&0&0&0&0 \\
0&0&0&0&0&0&1&1&1&1&0&0&0&0&0&0 \\
0&0&0&0&0&0&0&0&1&1&1&1&0&0&0&0 \\
0&0&0&0&0&0&0&0&0&0&1&1&1&1&0&0 \\
0&0&0&0&0&0&0&0&0&0&0&0&1&1&1&1 \\
1&0&1&0&1&0&1&0&1&0&1&0&1&0&1&0
\end{array} \right),
\qquad
\mathscr{K}_{16} = \left(
\begin{array}{*{16}{c}}
1&1&1&1&1&1&1&1&1&1&1&1&1&1&1&1 \\
0&2&0&0&0&0&0&0&0&0&0&0&0&0&0&2 \\
0&0&2&0&0&0&0&0&0&0&0&0&0&0&0&2 \\
0&0&0&2&0&0&0&0&0&0&0&0&0&0&0&2 \\
0&0&0&0&2&0&0&0&0&0&0&0&0&0&0&2 \\
0&0&0&0&0&2&0&0&0&0&0&0&0&0&0&2 \\
0&0&0&0&0&0&2&0&0&0&0&0&0&0&0&2 \\
0&0&0&0&0&0&0&2&0&0&0&0&0&0&0&2 \\
0&0&0&0&0&0&0&0&2&0&0&0&0&0&0&2 \\
0&0&0&0&0&0&0&0&0&2&0&0&0&0&0&2 \\
0&0&0&0&0&0&0&0&0&0&2&0&0&0&0&2 \\
0&0&0&0&0&0&0&0&0&0&0&2&0&0&0&2 \\
0&0&0&0&0&0&0&0&0&0&0&0&2&0&0&2 \\
0&0&0&0&0&0&0&0&0&0&0&0&0&2&0&2 \\
0&0&0&0&0&0&0&0&0&0&0&0&0&0&2&2
\end{array} \right)
\endgroup
\]

\end{document}